\documentclass[12pt]{amsart}
\usepackage{amscd}
\usepackage{verbatim}
\usepackage{amssymb, amsmath, amsthm, amscd,ifthen}
\usepackage[dvips]{graphics}
\usepackage[cp866]{inputenc}
\usepackage{graphicx}
\usepackage{epsfig}
\usepackage{color}
\usepackage{comment}

\usepackage{esdiff}
\usepackage{bbm}

\usepackage{amsmath,amssymb,amscd,}
\usepackage[all]{xy}

\usepackage[dvips]{graphics}
\usepackage[cp866]{inputenc}
\usepackage{graphicx}
\usepackage{epsfig}
\usepackage[hang]{subfigure}

\theoremstyle{plain}
\newtheorem{theorem}{Theorem}
\newtheorem{lemma}{Lemma}
\newtheorem{corollary}{Corollary}
\newtheorem{proposition}{Proposition}

\theoremstyle{definition}
\newtheorem{definition}{Definition}
\newtheorem{example}{Example}

\theoremstyle{remark}
\newtheorem{remark}[theorem]{Remark}

\theoremstyle{plain}
\newtoks\thehProclaim
\newtheorem*{Proclaim}{\the\thehProclaim}

\theoremstyle{definition}
\newtoks{\thehRemark}
\newtheorem*{Remark}{\the\thehRemark}

\renewcommand{\geq}{\geqslant}

\def\R{\mathbb{R}}
\def\D{\mathbb{D}}
\DeclareMathOperator\sign{sign}
\DeclareMathOperator\dist{dist}
\newcommand{\defeq}{:=}

\begin{document}

\title[Polygons with prescribed edge slopes]{Polygons with prescribed edge slopes: configuration space and extremal points of  perimeter}

\author{ Joseph Gordon, Gaiane Panina, Yana Teplitskaya}

\address[Joseph Gordon]{Mathematics \& Mechanics Department, St. Petersburg State University; e-mail: joseph-gordon@yandex.ru}
\address[Gaiane Panina]{Mathematics \& Mechanics Department, St. Petersburg State University; St. Petersburg Department of Steklov Mathematical Institute; e-mail: gaiane-panina@rambler.ru}
\address[Yana Teplitskaya]{Chebyshev Laboratory, St. Petersburg State University, 14th Line V.O., 29B, Saint Petersburg 199178 Russia; e-mail:
janejashka@gmail.com}



\begin{abstract}We describe the configuration space $\mathbf{S}$ of polygons with prescribed edge slopes, and study the
perimeter  $\mathcal{P}$  as a Morse function on $\mathbf{S}$. We characterize  critical points of  $\mathcal{P}$ (these are \textit{tangential} polygons) and compute their Morse indices.  This setup is motivated by a number of results about  critical points and Morse indices of the oriented area function defined on the configuration space of polygons with prescribed edge lengths (flexible polygons). As a by-product, we present an independent computation of the Morse index of the area function (obtained earlier by  G. Panina and A. Zhukova).
\end{abstract}

\maketitle

\section{Introduction}

  Consider the space $\mathbf{L}$ of planar polygons with prescribed edge lengths\footnote{The space $\mathbf{L}$ appears in the literature as ``configuration space of a flexible polygon'', or ``configuration space of a polygonal linkage'', or just as ``space of polygons''.}  and the oriented area $\mathcal{A}$  as a Morse function  defined on it.
  It is known that  generically:
  \begin{itemize}
    \item $\mathbf{L}$  is a smooth closed manifold whose diffeomorphic type depends on the edge lengths \cite{Farber, MilKap}.
    \item  The oriented area $\mathcal{A}$ is a Morse function whose critical points are cyclic configurations (that is, polygons with all the vertices lying on a circle), whose Morse indices are known, see Theorem \ref{Thm_Morse_closed_plane}, \cite{khipan,panzh, zhu}. The Morse index depends not only on the combinatorics of a cyclic polygon, but also on some metric data. Direct computations of the Morse index proved to be quite involved, so the existing proof  comes from bifurcation analysis combined with a number of combinatorial tricks.
    \item Bifurcations of $\mathcal{A}$ are captured by  cyclic polygons $P$ whose dual polygons $P^*$ have zero perimeter \cite{panzh}; see also Lemma \ref{LemmaBifurc}. That is, in a generic one-parametric family of edge lengths a critical point $P$ bifurcates whenever the perimeter of the dual tangential polygon $P^*$  vanishes.
  \end{itemize}
The polygon $P^*$ is \textit{tangential} (see Definition \ref{DefTang}),  so tangential polygons with zero perimeter play a special role in the framework of flexible polygons and oriented area. The initial motivation of the present paper was to clarify
this role.

  In the paper we consider the following problem: instead of prescribing edge lengths, we prescribe the slopes of the edges. Instead of taking the oriented area as a Morse function, we take the oriented perimeter. We prove:
 \begin{itemize}
    \item The space $\mathbf{S}$ of polygons with prescribed edge slopes  is a smooth non-compact  manifold (see Theorem \ref{ThmConfSpace} for its diffeomorphism type).
    \item The (oriented) perimeter $\mathcal{P}$ is a Morse function with either zero or  two critical points (Theorem \ref{ThmCritPer}).
    Critical points of $\mathcal{P}$ are{ tangential} polygons.
    \item  The absence of critical points is captured by existence of tangential polygons with zero perimeter (Corollary  \ref{CorCritPer}). That is, in a generic one-parametric family of slopes critical points disappear whenever the perimeter of the (uniquely defined) tangential polygon vanishes.

        \item  Although there are at most two critical points, these are not necessarily maximum and minimum of the perimeter function. The Morse index of a tangential polygon is expressed in Theorem \ref{ThmMorseTangential}. The proof is based on  direct computation of the leading principal minors of the Hessian matrix.
            \item The Morse index of a tangential polygon depends on the combinatorics of the polygon and the sign of its perimeter only.
            \item Local projective duality provides an alternative proof of Theorem \ref{Thm_Morse_closed_plane}, that is, the formula for the Morse index of a cyclic polygon (rel the area function).
  \end{itemize}

  Yet another motivation of this research is projective duality. Oversimplifying, assume that  the ambient space of the polygons is the sphere $S^2$.
Then projective duality takes polygons with prescribed edge lengths to polygons with prescribed angles. It also takes area to a linear function of perimeter, so the critical polygons in the two settings are mutually projectively dual and have related Morse indices.
(A necessary warning:  there exists only  a local version of projective duality. However it is sufficient for purposes.)

\section*{Acknowledgments}
Sections 3,4, and 5 are supported by the Russian Science Foundation grant N 14-21-00035.

Gaiane Panina is
 supported by the RFBR grant 17-01-00128  and   the Program of the Presidium of the Russian
 Academy of Sciences N 01 'Fundamental Mathematics and its Applications'
 under
 grant PRAS-18-01.

 We are indebted to Alexander Gaifullin who was the first to point out the vanishing perimeter of a bifurcating polygon.
We also thank Mikhail Khristoforov for useful discussions.

\section{Definitions and setups}

A \textit{polygon} is an oriented closed broken line in the plane. We assume that its vertices are numbered in the cyclic order and thus induce an orientation of the polygon.

\begin{definition} \label{Dfn_area} The \textit{oriented area} of a polygon $P$ with the vertices \newline $v_i = (x_i,
y_i)$  is defined by
$$2\mathcal{A}(P) = (x_1y_2 - x_2y_1) + \ldots + (x_ny_1 - x_1y_n).$$

Equivalently, one defines    $$\mathcal{A}(P):=\int_{\mathbb{R}^2} w(P,x)dx,$$

where $w(P,x)$  is the winding number of $P$ around  the point $x$.

\end{definition}

\subsection{Polygons with prescribed edge lengths. Oriented area as a Morse function.}  Assume that a generic $n$-tuple of positive numbers $(l_1,...,l_n)$ is given.
The space of all planar polygons whose consecutive edge lengths are  $(l_1,...,l_n)$ (modulo translations and rotations) is called the \textit{configuration space
of polygons with prescribed edge lengths.} We denote it by $\mathbf{L}=\mathbf{L}(l_1,...,l_n)$.

Generically, $\mathbf{L}$ is a smooth closed manifold  \cite{Farber, MilKap}, and
 the oriented area $\mathcal{A}$  is a  Morse function on $\mathbf{L}$.

\begin{definition}
    A
polygon  $P$  is  \textit{cyclic} if all its vertices $v_i$
lie on a circle.

\end{definition}

\begin{theorem}\label{Thm_crirical_are_cyclic}\cite{khipan}  Generically, $\mathcal{A}$ is a Morse function.
 At smooth points of the space  $\mathbf{L}$, a polygon $P$ is a critical point of the
oriented area  $\mathcal{A}$  iff $P$ is a cyclic configuration.
        \qed
\end{theorem}

Before we recall a formula for the Morse index of a cyclic configuration from \cite{zhu}, \cite{khipan1},
let us fix the following notation for a cyclic polygon $P$.

$\omega_P=w(P,O)$ is the winding number of $P$ with respect to the center $O$ of the circumscribed circle.

$\alpha_i$  is the half of the angle between the vectors
$\overrightarrow{Ov_i}$ and $\overrightarrow{Ov_{i+1}}$. The angle
is defined to be positive, orientation is not involved.

$\varepsilon_i$ is the
orientation of the edge $v_iv_{i+1}$, that is,

 $\varepsilon_i=\left\{
                       \begin{array}{ll}
                         1, & \hbox{if the center $O$ lies to the left of } \overrightarrow{v_iv_{i+1}};\\
                         -1, & \hbox{if  $O$ lies to the right of } \overrightarrow{v_iv_{i+1}}.
                       \end{array}
                     \right.$

$e(P)$ is the number of positive entries in $\varepsilon_1,...,\varepsilon_n$, that is, $e(P)$ is the number of \textit{positively oriented }edges .

$\mu_P=\mu_P(\mathcal{A})$ is the Morse index of the function $\mathcal{A}$ at the point
$P$. That is, $\mu_P(\mathcal{A})$ is the number of negative eigenvalues of the
Hessian matrix $Hess_P(\mathcal{A})$.

\begin{theorem}\label{Thm_Morse_closed_plane}\cite{khipan1}, \cite{panzh}, \cite{zhu} Generically,
for a cyclic polygon $P$,
$$\mu_P( \mathcal{A})= e(P)-1-2\omega_P-\left\{
       \begin{array}{ll}
    0    &\hbox{if }\   \sum_{i=1}^n \varepsilon_i \tan \alpha_i>0; \\
        1 & \hbox{otherwise}.\qed
       \end{array}
     \right.$$
\end{theorem}

In the present paper we give an alternative proof of this theorem: we prove a slightly stronger claim, see Corollary \ref{RemFinal}.
\begin{remark}
 In a continuous one-parametric family of  cyclic polygons  with non-vanishing edge lengths, $\mu_P$ changes iff  $\sum_{i=1}^n \varepsilon_i \tan \alpha_i$  vanishes. Although $e(P)$ and $\omega_P$  can vary, the sum $e(P)-1-2\omega_P$ is constant.
\end{remark}
\bigskip

\subsection{Polygons with prescribed edge slopes}

Fix $n$ pairwise non-parallel straight lines  \(s_1,\ldots,s_n\) in \(\R^2\) passing through the origin and call them \textit{slope lines}.

Each \(n\)-tuple of  lines \(e_1,\ldots,e_n\subset \R^2\)
 with  $e_i$ parallel to $s_i$ yields a polygon $P$ whose consecutive vertices  are \(v_1=e_1\cap e_2, \ldots, v_n=e_n\cap e_1 \). We will denote this polygon by \(Q=Q\left(e_1,\ldots,e_n\right)\)  and say that $Q$ is a \textit{polygon with edge slopes} \(s_1,\ldots,s_n\) .

The space of all  polygons $Q$ (modulo translations)
 with edge slopes \(s_1,\ldots,s_n\),
  is denoted  by $\widetilde{\mathbf{S}}=\widetilde{\mathbf{S}}(s_1,...,s_n)$.

The subspace of $\widetilde{\mathbf{S}}(s_1,...,s_n)$ consisting of polygons with $|\mathcal{A}(Q)|=1$
is called the    \textit{configuration space
of polygons  with prescribed edge slopes}. We denote it by $\mathbf{S}=\mathbf{S}(s_1,...,s_n)$. It splits into a disjoint union \(\mathbf{S}_-\sqcup \mathbf{S}_+\) where the index indicates the sign of \(\mathcal{A}\).

\begin{figure}[h]
\centering \includegraphics[width=10 cm]{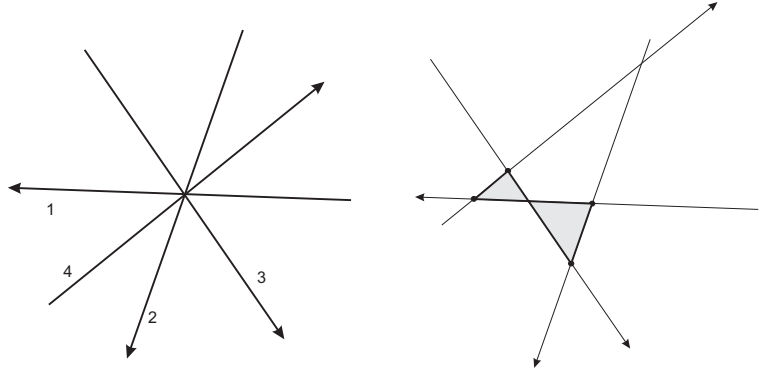}
\caption{Slope lines (left)  and  a polygon with these edge slopes (right).}\label{FigSlopes}
\end{figure}

Note that: (1) the condition $|\mathcal{A}(Q)|=1$ means that we  factor out dilations; (2) fixing slopes is the same as fixing angles and factoring out rotations.
\bigskip

Now fix a direction on each of the slope lines $\vec{s}_i$. Take $Q=Q(e_1,...,e_n)\in \mathbf{S}(s_1,...,s_n)$, and orient the lines $e_1,...,e_n$ consistently. Denote the oriented lines by $\vec{e}_1,\ldots,\vec{e}_n$.

\begin{definition}\label{DefPerim}
  The  perimeter of a polygon $Q\in \widetilde{\mathbf{S}}(s_1,...,s_n)$  is defined as follows:
\[
\mathcal{P}(Q)=\sum_{i=1}^n \sign_Q(i)|v_iv_{i+1}|,
\]

where \(\sign_Q(i)=\left\{
                                                   \begin{array}{ll}
                                                     1, & \hbox{if $\vec{e}_i$ is codirected with $\overrightarrow{v_iv}_{i+1}$;} \\
                                                     -1, & \hbox{otherwise.}
                                                   \end{array}
                                                 \right.\)

\end{definition}

Thus defined, perimeter may be negative or vanish.

\section{Topology of the configuration space  of polygons with prescribed slopes}

Let the \textit{angle \(\angle(r,s)\) between two lines} $r$ and $s$ be the minimal positive angle such that the counterclockwise rotation by \(\angle(r,s)\) takes \(r\) to \(s\).

Assuming that an $n$-tuple of slope lines  \(s_1,\ldots,s_n\) in \(\R^2\) is fixed, set
\newline \(t(s_1,\ldots,s_n)\defeq\sum_{i=1}^{n-1}\angle(s_i,s_{i+1})+\angle(s_n,s_1)\).

\medskip

The example for \(n=3\) will be useful in the sequel:
\begin{example}\label{triangle}
\(t(s_1,s_2,s_3)\) is either \(\pi\) or \(2\pi\). In the first (respectively, second) case the area of each nondegenerate triangle in \(\widetilde{\mathbf{S}}\) is negative (respectively, positive).

\end{example}

\begin{lemma}
\label{turn}
\(\,\)
\begin{enumerate}

    \item \(t(s_1,\ldots,s_n)\) takes values in $\{\pi,2\pi, \dots, (n-1)\pi\}$.
    \item \(t(s_1,\ldots,s_n)=t(s_1,\ldots,s_{n-1})+t(s_1,s_{n-1},s_n)-\pi\).\qed

\end{enumerate}

\end{lemma}

The informal meaning of the following lemma is: the area and perimeter behave additively with respect to homological sum of polygons.
\begin{lemma}
\label{sum}
\(\,\)
For a polygon  \(Q=Q(e_1,\ldots,e_n)\in\widetilde{\mathbf{ S}}(s_1,...,s_n)\), we have
\begin{enumerate}
    \item \(\mathcal{A}(Q)= \mathcal{A}(Q_{1,\ldots,n})=\mathcal{A}(Q_{1,\ldots,n-1})+\mathcal{A}(Q_{1,n-1,n})\);
		\item  \(\mathcal{P}(Q)= \mathcal{P}(Q_{1,\ldots,n})=\mathcal{P}(Q_{1,\ldots,n-1})+\mathcal{P}(Q_{1,n-1,n})\),

\end{enumerate}
where $Q_{i_1,...,i_k}=Q(e_{i_1},...e_{i_k})\in \widetilde{\mathbf{S}}(s_{i_1},...,s_{i_k})$.\qed
\end{lemma}

\begin{figure}[h]
\centering \includegraphics[width=8 cm]{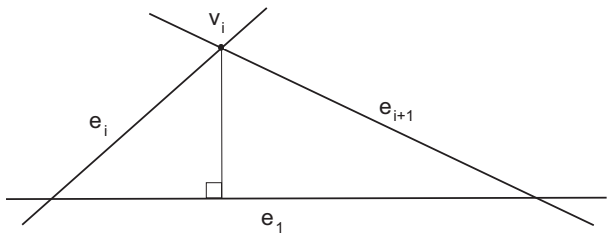}
\caption{A summand of $\mathcal{A}(Q_{1,\ldots,n})$.}\label{FigCoordinates}
\end{figure}

\begin{theorem}\label{ThmConfSpace}  The configuration space $\mathbf{S}(s_1,\ldots,s_n)$ is homeomorphic to disjoint union of products of a sphere and a disc:
$$
 S^{n-k-2}\times \D^{k-1}\sqcup S^{k-2} \times \D^{n-k-1}
,$$ where $t(s_1,\ldots,s_n)=k\pi$.

The left-hand  part corresponds to \(\mathbf{S}_-\) and the right-hand part corresponds to \(\mathbf{S}_+\).
\end{theorem}

\begin{proof}
We shall prove that \(\mathbf{S}_+\) is homeomorphic to \(S^{k-2} \times \D^{n-k-1}\). The proof  for \(\mathbf{S}_-\) is analogous.

By Lemma \ref{sum},
\[
\mathcal{A}(Q_{1,\ldots,n})=\sum_{i=2}^{n-1}\mathcal{A}(Q_{1,i,i+1})=\sum_{i=2}^{n-1}\sign\Big(t(s_1,s_i,s_{i+1})-\frac{3\pi}2\Big)\cdot c_i\cdot \dist(v_i,e_1)^2,
\]
where \(c_i=c_i(s_1,s_i,s_{i+1})\) is some positive constant depending only on the edge slopes (see Example \ref{triangle} and Fig. \ref{FigCoordinates}).  Thus we can parameterize $\widetilde{\mathbf{S}}(s_1,...,s_n)$ by \(\{x_i=\sqrt{c_{i+1}}\dist(v_{i+1},e_1)\}_{i=1}^{n-2}\). That is,
\[\mathbf{S}_+=\{x\in\R^{n-2} \mid \sum_{j\in A}x_j^2 - \sum_{j\in B} x_j^2 =1\},\] where
\(|A|=k-1\), \(A\sqcup B=\{1,\ldots,n-2\}\).
This is homeomorphic to \[\{x\in\R^{n-2} \mid \sum_{j\in A}x_j^2 =1, \sum_{j\in B}x_j^2 <1\}.\]
Indeed, for \(k>1\) the homeomorphism \(h: x\mapsto \frac{x}{\sqrt{\sum_{j\in A}x_j^2}}\) is well defined and appropriate, whereas for \(k=1\) both sets are empty.
The claim follows.
\end{proof}

\section{Critical points of the perimeter }
Assume that an $n$-tuple of directed slope lines  \(\vec{s}_1,\ldots,\vec{s}_n\) in \(\R^2\) is fixed.

\begin{definition}\label{DefTang} \begin{enumerate}
                                 \item A polygon \(Q(e_1,\ldots,e_n)\in \widetilde{\mathbf{S}}(s_1,...,s_n)\)  is \textit{tangential} if there exists a circle  $\sigma$ such that

                                      (a)  each of $e_i$  is tangent to $\sigma$, and

(b) either $\sigma$ lies on the left with respect to all of $\vec{e}_i$, or $\sigma$ lies on the right with respect to all of $\vec{e_i }$, see Fig. \ref{FigSuperscribed}.

In this case we say that the circle $\sigma$ is \textit{inscribed} in $Q$  and write $\sigma=\sigma(Q)$.
                                 \item By the \textit{radius} $r=r(\sigma(Q))$  of the inscribed circle $\sigma$ we mean the usual radius taken with the  sign  ''$+$'' if $\sigma$ lies on the left of $\vec{e}_i$, and with the sign ''$-$'' otherwise.
                               \end{enumerate}

\end{definition}

\begin{figure}[h]
\centering \includegraphics[width=14 cm]{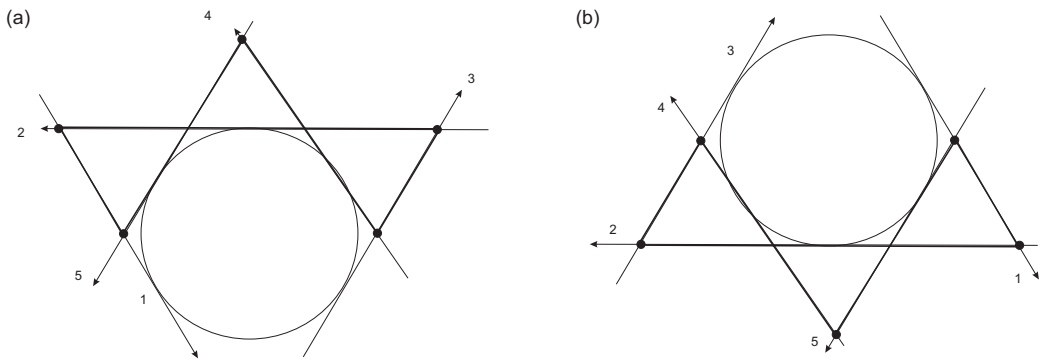}
\caption{(a) A tangential polygon $Q$ with the inscribed circle of positive radius. Here $\mathcal{A}(Q)<0, \mathcal{P}(Q)<0$.
(b) The tangential polygon $-Q$ with $\mathcal{A}(-Q)=\mathcal{A}(Q)<0, \mathcal{P}(-Q)=-\mathcal{P}(Q)>0$. The inscribed circle has negative radius. }\label{FigSuperscribed}
\end{figure}

\begin{lemma}\label{APr}
  The area  and perimeter of a tangential polytope $Q$ satisfy:
  $$\mathcal{A}(Q)=\frac12\mathcal{P}(Q)\cdot r(\sigma(Q)).$$

\end{lemma}

\begin{theorem}\label{ThmCritPer}
 \(Q\in\mathbf{S}\) is a critical point of \(\mathcal{P}\) iff $Q$ is tangential.
\end{theorem}

\begin{proof}

Let \(r_i\) be the radius of the (uniquely defined) circle inscribed\footnote{in the sense of Definition \ref{DefTang}.}  in the triangle \(Q_{1,i+1,i+2}\), \ \(i=1,\ldots,n-2\).  Denote by \(p_i\) the perimeter of the triangle homothetic to $Q_{1,i+1,i+2}$ (that is, with the same edge slopes) whose inscribed radius equals $1$. If $r_i\neq 0$, \(p_i\) is the perimeter of is the triangle
 \(\frac{1}{r_i}Q_{1,i+1,i+2}\).
Then by Lemma~\ref{sum} and Lemma~\ref{APr} we have
$$
                                  \mathcal{P}(Q)=\sum_{i=1}^{n-2}p_i\cdot  r_i,  \ \ \ \ \hbox{and}$$
                                 $$ \pm1=\mathcal{A}(Q)=\frac{1}{2}\sum_{i=1}^{n-2}p_i \cdot r_i^2.$$

The collection of radii gives a coordinate system on $\widetilde{\mathbf{S}}$.
W.l.o.g. we assume \(r_1>0\). Locally on $\mathbf{S}$, the second row implicitly defines a function \(r_1(r_2,\ldots,r_{n-2})\). Let us take \(j\)-th partial derivatives of the both rows, \(j=2,\ldots,n-2\).
 $$             \hbox{Since } \         0=\diffp{\mathcal{A}}{{r_j}}(Q)=p_1 r_1\diffp{r_1}{{r_j}}+p_j r_j,
$$

$$ \hbox{we have } \
                                  \diffp{r_1}{{r_j}}=-\frac{p_j r_j}{p_1 r_1} .$$
$$ \hbox{Therefore, } \
                                  \diffp{\mathcal{P}}{{r_j}}(Q)=p_1 \diffp{r_1}{{r_j}}+p_j=p_j(1-\frac{r_j}{r_1}). $$

Thus the gradient of the perimeter is zero iff \(r_1=r_2=\ldots=r_{n-2}=r\)
\end{proof}

There exists exactly one (up to a dilation) pair of mutually symmetric tangential polygons  $Q$ and $-Q$. For them we have: $\mathcal{A}(-Q)=\mathcal{A}(Q),\  \mathcal{P}(-Q)=-\mathcal{P}(Q)$, and $r(-Q)=-r(Q)$, see Fig.~\ref{FigSuperscribed} for example. If the area is non-zero, scaling gives $|\mathcal{A}|=1$.

\begin{corollary}\label{CorCritPer}
If the area of a tangential polygon is zero, there are no critical points of \(\mathcal{P}\) on the configuration space $\mathbf{S}$.
Otherwise there are exactly two critical points on the configuration space $\mathbf{S}$. Either they both lie in $\mathbf{S}_+$, or they both lie in $\mathbf{S}_-$.\qed

\end{corollary}

\medskip

A configuration space $\mathbf{S}(s_1,...,s_n)$ with no critical points is called\textit{ exceptional}.

\section{Morse index of a tangential polygon}
Now compute the Hessian matrix of $\mathcal{P}$ at a critical point. Assume that $Q$ is a tangential polygon. In notation of the previous section, we have:
\[
\diffp[2]{\mathcal{P}}{{r_j}}(Q)=p_1\diffp[2]{r_1}{{r_j}}=p_1\diffp{}{{r_j}}\Big(-\frac{p_j r_j}{p_1 r_1}\Big)=-p_j\Big(\frac1{r_1}-\frac{r_j}{r_1^2}\diffp{r_1}{{r_j}}\Big)=-\frac{p_j}{r_1}\Big(1+\frac{p_j r_j^2}{p_1 r_1^2}\Big).
\]

Since at the critical point all \(r_i\) are equal,
\[
H_{jj}(Q)=\diffp[2]{\mathcal{P}}{{r_j}}(Q)=-\frac{p_j}{r p_1}(p_1+p_j).
\]
In the same way, for \(j\neq k\)
\[
\diffp{\mathcal{P}}{{r_j}{r_k}}=p_1\diffp{r_1}{{r_j}{r_k}}=p_1\diffp{}{{r_k}}\Big(-\frac{p_j r_j}{p_1 r_1}\Big)=-p_j r_j\Big(-\frac1{r_1^2}\cdot \diffp{r_1}{{r_k}}\Big)=-\frac{p_j r_j}{r_1^2}\cdot \frac{p_k r_k}{p_1 r_1}.
\]

Thus: $$ H_{jk}(Q)=\diffp{\mathcal{P}}{{r_j}{r_k}}=-\frac{p_j p_k}{r p_1}.$$

To compute its determinant we do the following:
\begin{enumerate}
\item add all the columns to the first one;
\item subtract the first row from the \(i\)-th row $(i=2,3,...,n-2)$ taken with the coefficient \(\frac{p_{i+1}}{p_2}\) ;
\item subtract all rows from the first row with coefficient \(\frac{p_2}{p_1}\).
\end{enumerate}

We get:

\begin{equation}\label{det}
r^{n-3}\det(H)=
\begin{vmatrix}
-\frac{p_2}{p_1}\Pi & 0 & \ldots & 0 \\
0 & -p_3 & \ldots & 0 \\
& & \ddots \\
0 & \ldots & 0 & -p_{n-2}
\end{vmatrix},
\end{equation}

where \(\Pi=p_1+\ldots+p_{n-2}=\frac{\mathcal{P}(Q)}r\).

\begin{definition}
  For an ordered pair of  slopes $\vec{s}_i$ and $\vec{s}_j$ we say that we have the \textit{right turn}  (\textit{left turn}, respectively), if $\vec{s}_j$ is obtained from $\vec{s}_i$  by a clockwise (counterclockwise, respectively)  turn by an angle smaller than $\pi$, see Fig. \ref{FigTurn}.

  The\textit{ number of right turns} \(RT=RT(\vec{s_1},\ldots,\vec{s_n})\) for a slope collection $(\vec{s_1},\ldots,\vec{s_n})$ is
the number of right turns of the pairs $(\vec{s_1},\vec{s_2}),(\vec{s_2},\vec{s_3}),...,(\vec{s_{n}},\vec{s_1}) . $
The\textit{ number of left turns} \(LT\) is defined analogously.
\end{definition}

\begin{figure}[h]
\centering \includegraphics[width=8 cm]{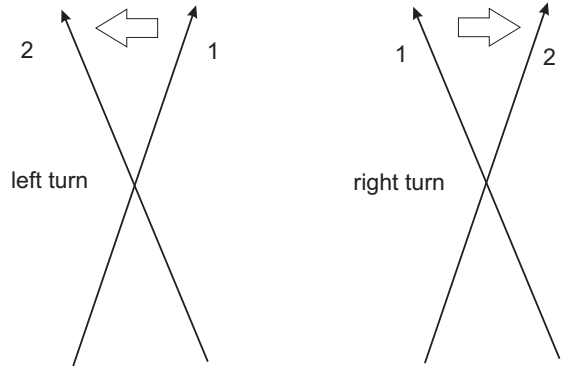}
\caption{Right and left turns}\label{FigTurn}
\end{figure}

\begin{theorem}\label{ThmMorseTangential}
Assume that for a tangential polygon $Q\in \mathbf{S}$,  the radius \(r\) of the inscribed circle $\sigma(Q)$ is positive.
{Then $Q$ is a Morse point, and its Morse index of $\mathcal{P}$ is equal to}
\begin{equation}\label{index}
\mu_Q(\mathcal{P})= RT-1+2\omega_Q-\left\{
       \begin{array}{ll}
    1    &\hbox{if }\   \mathcal{P}(Q)>0; \\
        0 & \hbox{otherwise}.
       \end{array}
     \right.
\end{equation}
In the case of negative radius it is equal to

\[
\mu_Q(\mathcal{P})=n-3-\mu_{-Q}(\mathcal{P})=LT-1-2\omega_Q-\left\{
       \begin{array}{ll}
    1    &\hbox{if }\   \mathcal{P}(Q)>0; \\
        0 & \hbox{otherwise}.
       \end{array}
     \right.
\]
\end{theorem}

\begin{example}
  For the polygon depicted in Fig. \ref{FigSuperscribed}  (a), we have $\omega=0$, $RT=2$, and $m=0$.

	For the polygon depicted in Fig. \ref{FigTanMorse}, we have $\omega=1$, $RT=1$, and $m=1$.
\end{example}

\begin{figure}[h]
\centering \includegraphics[width=8 cm]{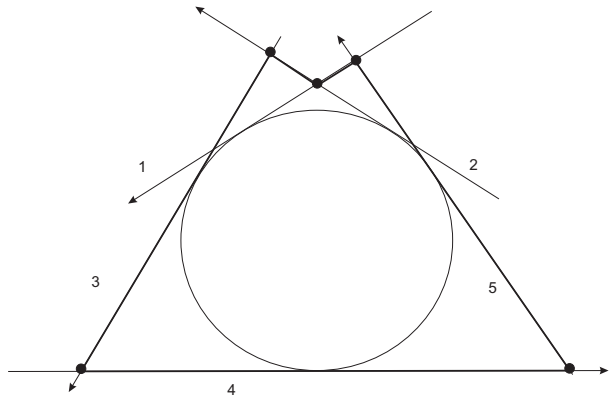}
\caption{This tangential  polygon is a saddle critical point of the perimeter.}\label{FigTanMorse}
\end{figure}
\begin{proof}

For now denote the right-hand side of (\ref{index}) as \(m(Q)\) and assume \(r>0\).

Let us prove (\ref{index}) by induction on \(n\).
For \(n=3\) the value of \(m\) is always zero, so is the Morse index.

Prove the claim for \(n+1\) assuming it is true for all the numbers smaller or equal than \(n\geq 3\).

Recall that the number of negative eigenvalues of the matrix equals the number of sign changes in the sequence of its leading principal minors. The \(k\)-th leading principal minor of \(H(Q)\) is the determinant of \(H(Q_{1,\ldots,k+2})\).

So we have:
\begin{itemize}
\item \(m(Q)=m(Q_{1,\ldots,n})+1\) whenever the sign of the determinant (\ref{det}) is different for \(n\) and \(n+1\);
\item \(m(Q)=m(Q_{1,\ldots,n})\) whenever the sign of the determinant (\ref{det}) is the same for \(n\) and \(n+1\).
\end{itemize}

The change of sign of the determinant (\ref{det}) depends only on the sign of \(p_{n-1}\) and the sign change of \(\Pi\).

Note that
\[
RT(s_1,\ldots,s_{n+1})=RT(s_1,\ldots,s_n)+RT(s_1,s_n,s_{n+1})-1
\]
and
\[
\omega_{Q_{1,\ldots,n+1}}=\omega_{Q_{1,\ldots,n}}+\omega_{Q_{1,n,n+1}}.
\]

Therefore,

\[
m(Q_{1,\ldots,n+1})-m(Q_{1,\ldots,n})=m(Q_{1,n,n+1})+\mathbbm{1}_{\mathcal{P}(Q_{1,n,n+1})>0}-\mathbbm{1}_{\mathcal{P}(Q)>0}+\mathbbm{1}_{\mathcal{P}(Q_{1,\ldots,n})>0}.
\]

The first summand is zero by the base of induction. So we get:
\[
m(Q_{1,\ldots,n+1})-m(Q_{1,\ldots,n})=\mathbbm{1}_{p_{n-1}>0}-\mathbbm{1}_{p_1+\ldots+p_{n-1}>0}+\mathbbm{1}_{p_1+\ldots+p_{n-2}>0},
\]

which is exactly what we require.

The case of \(r<0\) follows from \(\mathcal{P}(-Q)=-\mathcal{P}(Q)\).

\end{proof}

\section{Cyclic polygons and tangential polygons meet}
Theorem \ref{Thm_Morse_closed_plane} motivates the following definition:

\begin{definition}
A cyclic polygon is a \textit{bifurcating polygon} if $\sum_{i=1}^n \varepsilon_i \tan \alpha_i=0$.
\end{definition}
Bifurcating polygons correspond to bifurcations of the area function.

\begin{definition}
Given a cyclic polygon $P=\{v_1,...,v_n\}$, define its \textit{dual  polygon} $P^*$ (Fig. \ref{FigRelation}) as a closed broken line with orientations on the edges  constructed as follows:

\begin{enumerate}
  \item  Take the lines $e_1,...,e_n$ tangential to the circle at the points $\{v_1,...,v_n\}$.
  \item Take the intersection points of $e_i$ and $e_{i+1}$.
  \item  Orient each of the lines so that the circle lies to the left of the line.
\end{enumerate}

The construction can be easily reversed: given a tangential polygon $Q$,  there exists a cyclic polygon $P$ such that $P^*=Q$.
\end{definition}

\begin{lemma}\label{LemmaBifurc}  \begin{enumerate}
                 \item For a cyclic polygon $P$, the perimeter of $P^*$ (in the sense of Definition \ref{DefPerim}) equals $\sum_{i=1}^n \varepsilon_i \tan \alpha_i$.
                 \item In particular, the perimeter of $P^*$ vanishes iff $P$ is a bifurcating polygon.
                 \item  The oriented area of $P^*$ vanishes iff $P$ is a bifurcating polygon.\qed
               \end{enumerate}
\end{lemma}

To summarize, if a polygon $P$ is a bifurcation of the area, then its dual $Q=P^*$  yields an exceptional configuration space $\mathbf{S}(s_1,...,s_n)$.

\begin{figure}[h]
\centering \includegraphics[width=8 cm]{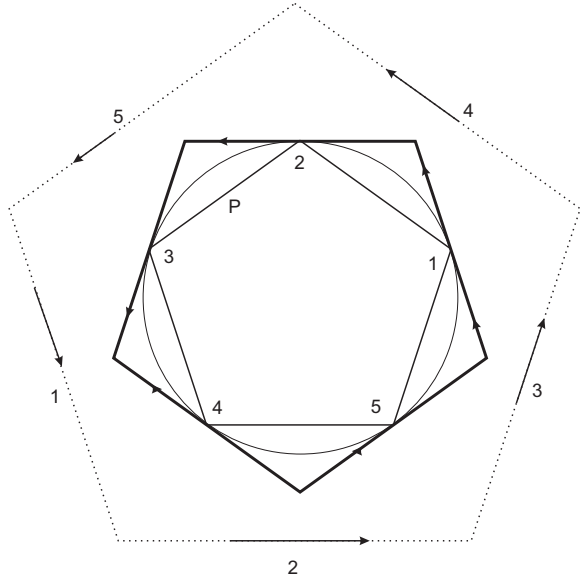}
\caption{ A cyclic polygon $P$, the polygon $P^*$ (bold), and the projectively dual polygon $P^o$ (dashed).}\label{FigRelation}
\end{figure}

\section{Shperical polygons, local projective duality, and an alternative proof of Theorem \ref{Thm_Morse_closed_plane}}

\subsection{Shperical polygons}

A \textit{spherical polygon} is an oriented closed broken line lying on the sphere of radius $R$. We always assume that its edges are the unique shortest geodesics, that is,  $l_i<\pi R$, so a spherical polygon is uniquely defined by the (circular) sequence of its vertices.

One fails to  correctly define the area function on the space of spherical polygons with  prescribed edge lengths. One also fails to define  the space of spherical polygons with  prescribed angles  together with the perimeter function.  However we shall make use of their local versions.

Here is how it goes:

\begin{definition}
  Let $P_0, Q_0$ be  spherical polygons.
  \begin{enumerate}
    \item Consider the space of all spherical polygons with the same edge lengths lying close\footnote{With respect to any reasonable metric} to $P_0$ subject to rotation of the sphere.  This space is  called the \textit{local configuration space }\textit{of spherical polygons with prescribed edge lengths} $\mathbf{L}^{loc}(P_0)$.
    \item Elimination of a point from the sphere allows to define the winding numbers for curves  in the sphere. So fix a point $\infty\notin P_0$ (that is, not lying on the broken line) and define the oriented area  of $P\in \mathbf{L}^{loc}(P_0)$  as the integral
        $$A(P):=\int_{S^2} w(P,x)dx.$$
    \item Analogously, we define the \textit{local configuration space of spherical polygons with prescribed angles} $\mathbf{S}^{loc}(Q_0)$. Once we set  some fixed orientations on the edges, we have a well defined perimeter function $\mathcal{P}$  on the space $\mathbf{S}^{loc}(Q_0)$. Note that the area $\mathcal{A}$ is constant on $\mathbf{S}^{loc}(Q_0)$.
  \end{enumerate}
\end{definition}

\begin{proposition}\label{PropSphCrit}
  Assume that a spherical polygon $P$ fits in a hemisphere not containing $\infty$. $P$  is a critical point of the area function $\mathcal{A}$ iff it is a cyclic polygon, that is, its vertices lie on a circle.
\end{proposition}
Proof.
(1) Prove first the statement for  polygons with four edges.
If a polygon bounds a spherically convex  region, then the statement is classical: a cyclic convex  $4$-gon exhibits either the maximum or the minimum point of the area, depending on orientation of $P$.

(''If'') Assume  that $P$ intersects itself  and is a cyclic polygon with vertices $1,2,3,4$. Add a new point $5$ on the circle  together with two new bars as is shown in Fig. \ref{fourgon}.  Then $\mathcal{A}(1234)=\mathcal{A}(1254)-\mathcal{A}(2543)$.  A local flex of the polygon $P$ induces flexes of the polygons
$(1254)$ and $(2543)$. Since the latter are critical, the claim follows.

(''Only if'') Assume that  $P$ intersects itself  and is a  critical point of $\mathcal{A}$, but not a cyclic polygon.  Take a circle superscribing $412$ and add  a new point $5$ on the circle  together with two new bars as we did above.  Now $(125 4)$ is a critical polygon, and  $(2543)$  is not.
 $\mathcal{A}(1234)=\mathcal{A}(1254)-\mathcal{A}(2543)$ completes the proof.

(2) The general case (any number of edges) is obtained by verbatim repeating  the reasonings from \cite{khipan}.\qed
\begin{figure}[h]
\centering \includegraphics[width=4 cm]{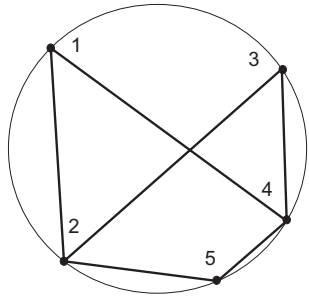}
\caption{Notation for Proposition \ref{PropSphCrit}.}\label{fourgon}
\end{figure}

\subsection*{Local duality of $\mathbf{L}^{loc}(P)$  and $\mathbf{S}^{loc}(P^o)$. }
Assume that $P$ is a spherical cyclic polygon. It fits in a hemisphere, and so do all the polygons from $\mathbf{L}^{loc}(P)  $.
Assume also that a point  $\infty$ lies beyond the hemisphere, so $\mathcal{A}$ is well-defined on $\mathbf{L}^{loc}(P)  $,
and $P$ is a critical point.

\begin{definition}
\begin{enumerate}
  \item In the above setting, define
the dual  polygon $P^0$:

\begin{enumerate}
\item Fix the hemisphere containing P centered at the center of the superscribed circle.
  \item Assume that $P=\{p_1,...,p_n\}$. Take the big circles $C_1,...,C_n$  projectively dual to the points $\{p_1,...,p_n\}$.
  \item Connect the intersection points $q_i$ of $l_i$ and $l_{i+1}$ lying in the hemisphere  by short geodesics.
  \item  Orient each of the lines such that the circle lies on the  {left} from each of the lines.
\end{enumerate}

  \item Continuously extend the duality to $\mathbf{L}^{loc}(P)  $. The extension  is uniquely defined by the condition that edges of the dual polygon lie on dual lines to the vertices of the initial polygon.
\end{enumerate}

\end{definition}
\begin{lemma}\label{LemmaRelAP}For a cyclic $P$,
  \begin{enumerate}
    \item The polygon $P^o$ is  tangential.
  \item \(\mu_{P}(\mathcal{A})=n-3-\mu_{P^o}(\mathcal{P}).\)

  \end{enumerate}
\end{lemma}
Proof. (1) is straightforward.

  Projective duality (on the unit sphere) takes edge lengths of a polygon to the exterior angles of the dual, and vice versa.
  Since
  $$\mathcal{A}(P)=Const -R\cdot \mathcal{P}(P^o),$$
  the claims (2) follows.\qed

\bigskip

\subsection*{Morse indices: planar vs spherical.}

Let $P_0$ be a planar cyclic polygon. It is uniquely defined by the circumscribed circle $\sigma$ and the ordered sequence of its points. Put the circle $\sigma$ with the $n$ points on the sphere of radius $R$, provided that $R>r(\sigma)$. It defines a spherical cyclic polygon $P_0^R$ fitting in the hemisphere centered at the center of  $\sigma$. In turn, the spherical polygon $P_0^R$  gives rise to the local configuration space $\mathbf{L}^{loc}(P_0^R)$ of polygons  with prescribed edge lengths. Clearly, the bigger $R$ is, the smaller is the distortion of edge lengths and angles.

\begin{lemma}\label{LemmaBifurkFree}
The Morse index of  the spherical polygon $P_0^R$ with respect to the area function $\mathcal{A}_R$ is the same as the Morse index of $P_0$.
\end{lemma}

 {Proof.
We have the one parametric family of spherical polygons and their local configuration spaces $P_0^R,\mathbf{L}^{loc}(Q_0^R)$. The area function $\mathcal{A}_R$  is well defined  on $\mathbf{L}^{loc}(P_0^R)$, and $P_0^R$ is its critical point.   As the radius $R$ tends to infinity, the polygon  $P_0^R$ deforms and tends to $P_0$.  Since $P_0^R$ is  the unique critical point in the neighborhood, $P_0^R$ does not bifurcate, so its Morse index does not change.  Besides, by standard arguments, the Morse index of $P_0^R$  converges to the Morse index of the planar polygon $P_0$. \qed

\bigskip

\bigskip
Now we are ready to prove  Theorem \ref{Thm_Morse_closed_plane}. Take a planar cyclic polygon $P$.
Take $P^R$ for some big $R$.  By the above lemma, $\mu_{P^R}(\mathcal{A}_R)= \mu_{P}(\mathcal{A})$. Gradually make $R$ smaller, such that the circumscribed circle tends to an equator of the sphere. The Morse index stays the same. Now take the projectively dual polygon  $Q^R:=(P^R)^o$. It is a small (and therefore, almost planar) tangential polygon.

By Lemma \ref{LemmaRelAP}, $$\mu_{P^R}(\mathcal{A}_R)=n-3-\mu_{Q^R}(\mathcal{P}_R).$$
Replace  $Q^R$ by a planar polygon $Q$. On the one hand, the Morse index $\mu_{Q}(\mathcal{P})$ stays the same. On the other hand, we know the Morse index by
Theorem \ref{Thm_Morse_closed_plane}. It remains to observe  that
  local  projective duality  maintains the winding number, and
               takes  {left turns} to positively oriented edges.
            \qed

This approach also gives the following fact which exceeds Theorem \ref{Thm_Morse_closed_plane}:
\begin{corollary}\label{RemFinal}
Assume we have a cyclic polygon $P$ such that (1) no two consecutive vertices are antipodal (with respect to the superscribed circle), and (2) the polygon does fit in a straight line. Then
$P$ is a Morse point of the oriented area function iff it is not a bifurcating polygon.
\end{corollary}
Proof. The dual polygon is a non-degenerate Morse point, see Theorem \ref{ThmMorseTangential}. \qed


\begin{thebibliography}{99}




\bibitem{Farber} M. Farber, Invitation to Topological Robotics, Zuerich Lectures in
Advanced Mathematics, European Mathematical Society (EMS), Zuerich,
2008.

\bibitem{MilKap} M. Kapovich, J. Millson, On the moduli space of polygons in the Euclidean
plane,  J. Differential Geom.  42, 1 (1995)  133-164.

\bibitem{khipan}
 G. Khimshiashvili, G. Panina, Cyclic polygons are critical
points of area,  Zap. Nauchn. Sem. S.-Peterburg. Otdel. Mat. Inst.
Steklov. (POMI), 360, 8 (2008) 238--245.


\bibitem{khipan1}G.
Khimshiashvili, G. Panina, On the Area of a Polygonal
Linkage, Dokl. Akad. Nauk, Mathematics, , Vol. 85, No. 1,
120-121(2012).

\bibitem{panzh}
G. Panina, A. Zhukova, Morse index of a cyclic polygon, Cent.
Eur. J. Math., 9,  2 (2011)  364--377.

\bibitem{zhu} A. Zhukova, Morse index of a cyclic polygon II, St. Petersburg
Math. J. 24 (2013) 461--474.

\end{thebibliography}
\end{document}